\documentclass[a4paper,12pt]{article}

\usepackage{amsmath}
\usepackage{amssymb}
\usepackage{amsthm}
\usepackage[dvipdfmx]{graphicx}
\usepackage{amscd}
\newcommand{\free}{$\mathbb{Z}^{2}*\mathbb{Z}$ }
\newcommand{\freen}{$\mathbb{Z}^{n}*\mathbb{Z}$ }

\title{Embeddings of right-angled Artin groups into higher dimensional Thompson groups}

\author{Motoko Kato}

\date{}

\begin{document}

\numberwithin{equation}{section}
\newtheorem{theorem}{Theorem}[section]
\newtheorem{proposition}[theorem]{Proposition}
\newtheorem{formula}{Formula}[section]
\newtheorem{lemma}[theorem]{Lemma}
\newtheorem{corollary}[theorem]{Corollary}
\newtheorem{remark}{Remark}[section]
\newtheorem{definition}{Definition}[section]

\maketitle

\begin{abstract}
  In this paper, we construct embeddings of right-angled Artin groups into higher dimensional Thompson groups.
In particular, we embed every right-angled Artin groups into $n$-dimensional Thompson groups, where $n$ is the number of complementary edges in the defining graph.
It follows that $\mathbb{Z}^{n}*\mathbb{Z}$ embeds into $nV$ for every $n\geq 1$.
\end{abstract}

\section{Intoduction}
The Thompson group $V$ is an infinite simple finitely presented group, which is described as a subgroup of the homeomorphism group of the Cantor set $C$.
Brin \cite{B} defined higher dimensional Thompson groups as generalizations of the Thompson group $V=1V$.
By definition, $n$-dimensional Thompson group $n_1V$ embeds into $n_2V$ when $n_1\leq n_2$.
Brin \cite{B} showed that $V$ and $2V$ are not isomorphic.
Bleak and Lanoue \cite{BL} showed $n_1V$ and $n_2V$ are isomorphic if and only if $n_1=n_2$.

In \cite{BS}, Bleak and Salazar-D\'\i az proved that \free does not embed in $V$.
Recently,  Corwin and Haymaker \cite{CH} determined which right-angled Artin groups embed into $V$. 
Using the nonembedding result of \cite{BS}, they showed that \free is the only obstruction for a right-angled Artin group to be embedded into $V$.
On the other hand, Belk, Bleak and Matucci \cite{BBM} proved that a right-angled Artin group embeds in $nV$ with sufficiently large $n$.
They took $n$ to be the sum of the number of vertices and the number of complementary edges in the defining graph.
They conjectured that a right-angled Artin group embeds into $(n-1)V$ if and only if the right-angled Artin group does not contain \freen.
Corwin \cite{C} constructed embeddings of \freen into $nV$ for every $n\geq 2$. It follows that every $nV$ with $n\geq2$ does not embed into $V$. 

In this paper, we give another construction of embeddings of right-angled Artin groups into higher-dimensional Thompson groups.
In particular, we may embed a right-angled Artin group into $nV$, where $n$ is the number of complementary edges in the defining graph. 
We may construct embeddings of \freen into $nV$ in this way.

The author would like to thank Takuya Sakasai and Tomohiko Ishida for helpful comments.
This work was supported by the Program for Leading Graduate 
Schools, MEXT, Japan.

\section{Right-angled Artin groups}\label{RAAG}

Let $\Gamma$ be a finite graph with a vertex set $V(\Gamma)=\{v_i\}_{1\leq i\leq m}$ and an edge set $E(\Gamma)$.
The corresponding right-angled Artin group, denoted by $A_\Gamma$, is a group defined by the presentation
\begin{align*}
A_\Gamma=\langle g_1,\ldots, g_m\mid g_ig_j=g_jg_i \text{ for all }\{v_i, v_j\}\in E(\Gamma)\rangle.
\end{align*}
In the following, we let 
\begin{align*}
\bar{E}(\Gamma)=\{\{v_i, v_j\}\mid v_i\not=v_j\in V(\Gamma) \text{ are not connected by edges.}\}
\end{align*}
We call the elements of $\bar{E}(\Gamma)$ complementary edges.

We use the following theorem, known as the ping-pong lemma for the right-angled Artin groups.

\begin{theorem}[\cite{K}]\label{PP}
Let $A_\Gamma$ be a right-angled Artin group with generators $\{g_i\}_{1\leq i\leq m}$ acting on a set $X$.
Suppose that there exist subsets $S_i$ $(1\leq i\leq m)$ of $X$, with divisions  $S_i=S_i^+\coprod S_i^-$, satisfying the following conditions:
\begin{itemize}
\item[$(1)$] $g_i(S_i^+)\subset S_i^+$ and $g_i^{-1}(S_i^-)\subset S_i^-$ for all $i$. 
\item[$(2)$] If $g_i$ and $g_j$ commute, then $g_i(S_j)=S_j$.
\item[$(3)$] If $g_i$ and $g_j$ do not commute, then $g_i(S_j)\subset S_i^+$ and $g_i^{-1}(S_j)\subset S_i^-$.
\item[$(4)$] There exists $x\in X-\bigcup_{i=1}^mS_i$ such that $g_i(x)\in S_i^+$ and $g_i^{-1}(x)\in S_i^-$ for all $i$.
\end{itemize}
Then this action is faithful.
\end{theorem}
\section{Embedding right-angled Artin groups into $nV$}

We use the following notations in \cite{B}.
We let $I$ be a half-open interval $[0,1)$.
An {\it $n$-dimensional rectangle} is an affine copy of $I^n$ in $I^n$, constructed by repeating ``dyadic divisions''.
An {\it $n$-dimensional pattern} is a finite set of $n$-dimensional rectangles, with pairwise disjoint, non-empty interiors and whose union is $I^n$.
A {\it numbered pattern} is a pattern with a one-to-one
correspondence to $\{0, 1, \ldots , r-1\}$ where $r$ is the number of rectangles in the pattern.

Let $P=\{P_i\}_{0\leq i\leq r-1}$ and $Q=\{Q_i\}_{0\leq i\leq r-1}$ be numbered patterns.
We define $v(P,Q)$ to be a map from $I^n$ to itself which takes each $P_i$ onto $Q_i$ affinely so as to preserve the orientation.
The {\it $n$-dimensional Thompson group} $nV$ is the set of partially affine, partially orientation preserving right-continuous bijections from $I^n$ to itself.

Using these notations, we give a construction of embeddings of right-angled Artin groups into higher dimensional Thompson groups.

\begin{theorem}\label{Main}
Let $\Gamma$ be a graph with the vertex set $V(\Gamma)=\{v_i\}_{1\leq i\leq m}$.
Suppose that there are nonempty subsets $\{D_i\}_{1\leq i\leq m}$ of $\{1,\ldots n\}$, 
such that $D_i\cap D_j=\emptyset$ if and only if $v_i$ and $v_j$ are connected by an edge.
Then the right-angled Artin group $A_\Gamma$ embeds into $nV$.
\end{theorem}

For a nonempty subset $D$ of $\{1,\ldots, n\}$, a {\it $D$-slice} of $I^n$ is an $n$-dimensional rectangle $S=\prod_{d=1}^nI_d$,
where $d\in D$ if and only if $I_d$ is properly contained in $[0,1)$.

\begin{lemma}\label{D}
Let $D$ be a nonempty subset of \{1,\ldots, n\}.
For every $D$-slice $S$ of $I^n$ and every division $S=S^+\coprod S^-$ where $S^+$ and $S^-$ are again $D$-slices,
there is $h\in nV$ satisfying 
\begin{itemize}
\item[$(1)$] $h$ changes $d$-th coordinate of $I^n$ if and only if $d\in D$.
\item[$(2)$] $h(I^n-S^-)=S^+$ and $h^{-1}(I^n-S^+)=S^-$.
\end{itemize} 
\end{lemma}

\begin{proof}
There is an $n$-dimensional pattern which contains $S$ as a rectangle and consists of $D$-slices.
We fix one of such pattern $P$, and consider $I^n-S$ as a disjoint union of $(|P|-1)$-many $D$-slices.

We divide $S^+$ into mutually disjoint $|P|$-many $D$-slices.
We choose one of those $D$-slices in $S^+$ and name it $S^{++}$.
We consider $S^+-S^{++}$ as a disjoint union of $(|P|-1)$-many $D$-slices.
Similarly, we choose a $D$-slices $S^{--}$ in $S^-$,
and consider $S^--S^{--}$ as a disjoint union of $(|P|-1)$-many $D$-slices.

We define $h\in nV$ as follows:
\begin{itemize}
\item[$0.$] $h$ maps $I^n-S$ to $S^+-S^{++}$.
\item[$1.$] $h$ maps $S^+$ to $S^{++}$.
\item[$2.$] $h$ maps $S^--S^{--}$ to $I^n-S$.
\item[$3.$] $h$ maps $S^{--}$ to $S^-$.
\end{itemize}
This $h$ satisfies conditions $(1)$ and $(2)$.
\end{proof}

We show the construction of the map $h$ in the following figure, 
in the case where $n=2$, $D=\{1,2\}$, $S=[0,1/2)\times[0,1/2)$ and $S^+=[0,1/4)\times[0,1/2)$.

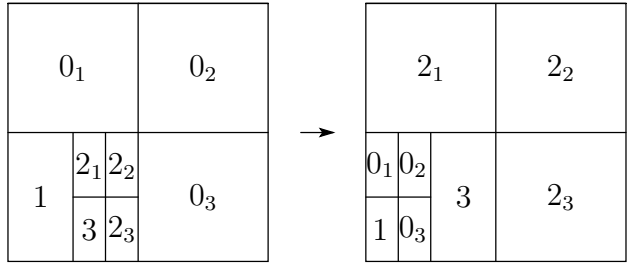
\begin{center}
{\unitlength 0.1in%
\begin{picture}(32.0000,13.4700)(4.0000,-17.4700)%
%
\special{pn 8}%
\special{pa 400 400}%
\special{pa 1747 400}%
\special{pa 1747 1747}%
\special{pa 400 1747}%
\special{pa 400 400}%
\special{pa 1747 400}%
\special{fp}%
%
\special{pn 8}%
\special{pa 1747 1074}%
\special{pa 400 1074}%
\special{fp}%
\special{pa 1074 400}%
\special{pa 1074 400}%
\special{fp}%
%
\special{pn 8}%
\special{pa 1074 400}%
\special{pa 1074 1747}%
\special{fp}%
%
\special{pn 8}%
\special{pa 737 1747}%
\special{pa 737 1074}%
\special{fp}%
%
\special{pn 8}%
\special{pa 737 1411}%
\special{pa 1074 1411}%
\special{fp}%
%
\special{pn 8}%
\special{pa 905 1747}%
\special{pa 905 1074}%
\special{fp}%
\put(7.3700,-7.3700){\makebox(0,0){$0_1$}}%
\put(14.1100,-7.3700){\makebox(0,0){$0_2$}}%
\put(14.1100,-14.1100){\makebox(0,0){$0_3$}}%
\put(9.8900,-12.5100){\makebox(0,0){$2_2$}}%
\put(8.2100,-12.5100){\makebox(0,0){$2_1$}}%
\put(9.8900,-15.7900){\makebox(0,0){$2_3$}}%
\put(8.2100,-15.7900){\makebox(0,0){$3$}}%
%
\put(4.8400,-14.1100){\makebox(0,0)[lb]{}}%
\put(5.6800,-14.1100){\makebox(0,0){$1$}}%
%
\special{pn 8}%
\special{pa 2253 400}%
\special{pa 3600 400}%
\special{pa 3600 1747}%
\special{pa 2253 1747}%
\special{pa 2253 400}%
\special{pa 3600 400}%
\special{fp}%
%
\special{pn 8}%
\special{pa 2926 1747}%
\special{pa 2926 400}%
\special{fp}%
%
\special{pn 8}%
\special{pa 2253 1074}%
\special{pa 3600 1074}%
\special{fp}%
%
\special{pn 8}%
\special{pa 2589 1074}%
\special{pa 2589 1747}%
\special{fp}%
%
\special{pn 8}%
\special{pa 2589 1411}%
\special{pa 2253 1411}%
\special{fp}%
%
\special{pn 8}%
\special{pa 2421 1074}%
\special{pa 2421 1747}%
\special{fp}%
\put(23.2800,-15.7900){\makebox(0,0){$1$}}%
\put(24.9700,-15.7900){\makebox(0,0){$0_3$}}%
\put(24.9700,-12.4200){\makebox(0,0){$0_2$}}%
\put(23.2800,-12.4200){\makebox(0,0){$0_1$}}%
%
\put(23.2800,-7.3700){\makebox(0,0)[lb]{}}%
\put(25.8900,-7.3700){\makebox(0,0){$2_1$}}%
\put(32.6300,-7.3700){\makebox(0,0){$2_2$}}%
\put(32.6300,-14.1100){\makebox(0,0){$2_3$}}%
\put(27.5800,-14.1100){\makebox(0,0){$3$}}%
%
\special{pn 8}%
\special{pa 1916 1074}%
\special{pa 2084 1074}%
\special{fp}%
\special{sh 1}%
\special{pa 2084 1074}%
\special{pa 2017 1054}%
\special{pa 2031 1074}%
\special{pa 2017 1094}%
\special{pa 2084 1074}%
\special{fp}%
\end{picture}}%
 
\end{center}

\begin{remark}\label{Rmk}
We take $h\in nV$ as in Lemma \ref{D} with respect to a $D$-slice $S$ and some division $S=S^+\coprod S^-$.
Let $S'$ be a $D'$-slice with $D\cap D'=\emptyset$. 
We may observe that $h(S')=S'$, because $S'$ is determined only by $d'$-th coordinates for $d'\in D'$, 
which are unchanged by $h$.
\end{remark}

\begin{lemma}\label{D2}
For nonempty subsets $\{D_i\}_{1\leq i\leq m}$of $\{1,\ldots, n\}$, 
there is a set of $n$-dimensional rectangles $\{S_i\}_{1\leq i\leq m}$ satisfying
\begin{itemize}
\item[$(1)$] For every $i$, $S_i$ is a $D_i$-slice of $I^n$.
\item[$(2)$] $S_i\cap S_j=\emptyset$ if and only if $D_i\cap D_j\not= \emptyset$. 
\item[$(3)$] $\bigcup_{i=1}^mS_i\subsetneqq I^n$.
\end{itemize}
\end{lemma}

\begin{proof}
We fix a dyadic division $I=\coprod_k J_k$, where $k\geq m+1$.
We define $S_i=\prod_{d=1}^n I_d^i$ 
by setting $I_d^i=J_i$ when $d\in D_i$, and $I_d^i=I$ otherwise.

$(1)$ Such $S_i$ is a $D_i$-slice.

$(2)$ If $D_i\cap D_j\not=\emptyset$, then $S_i\cap S_j= \emptyset$ 
since $I_d^i\cap I_d^j=J_i\cap J_j=\emptyset$ for all $d\in D_i\cap D_j$.
The converse follows from the observation that a $D$-slice and a $D'$-slice always intersect when $D\cap D'=\emptyset$.

$(3)$ Since we took $J_k$ small enough, $\bigcup_{i=1}^m S_i$ is properly contained in $I^n$.

Therefore, $\{S_i\}_{1\leq i\leq m}$ satisfies conditions required in Lemma \ref{D2}.

\end{proof}

\begin{proof}[Proof of Theorem \ref{Main}]
Let $\Gamma$ be a finite graph with vertices $\{v_i\}_{1\leq i\leq m}$. Let $\{D_i\}_{1\leq i\leq m}$ be nonempty subsets of $\{1,\ldots, n\}$ such that $D_i\cap D_j=\emptyset$ if and only if $v_i$ and $v_j$ are connected by an edge.

According to Lemma \ref{D2}, we take $\{S_i\}_{1\leq i\leq m}\subset I^n$ with respect to $\{D_i\}_{1\leq i\leq m}$.
For every $i$, we fix $D_i$-slices $S_i^+$ and $S_i^-$ satisfying $S_i=S_i^+\coprod S_i^-$.
We define $h_i$ to be $h$ of Lemma \ref{D}, which is defined with respect to $S_i$, $S_i^+$ and $S_i^-$.


We may define a homomorphism $\phi: A_\Gamma \to nV$ which maps each generator $g_i$, 
corresponding to the vertex $v_i$, to $h_i$.
This homeomorphism is well-defined, since $h_i$ and $h_j$ commute when $v_i$ and $v_j$ are connected by an edge, according to the first condition of Lemma \ref{D}.

We consider an action of $A_\Gamma$ on $I^n$, which is defined by $g\cdot x=\phi(g)(x)$.
In the following, we show that this action is faithful, and thus $\phi$ is injective.      

$(1)$ By the definition of $h_i$, 
$h_i(S_i^+)\subset S_i^+$ and $h_i^{-1}(S_i^-)\subset S_i^-$ for all $i$.

$(2)$ According to Remark \ref{Rmk}, $h_i(S_j)=S_j$ when $g_i$ and $g_j$ commute.

$(3)$ When $g_i$ and $g_j$ do not commute, $v_i$ and $v_j$ are not connected by an edge, and $S_i$ and $S_j$ are disjoint.
Therefore $h_i(S_j)\subset h_i(I^n-S_i)\subset S_i^+$ and $h_i^{-1}(S_j)\subset h_i^{-1}(I^n-S_i)\subset S_i^-$.

$(4)$ Since $\bigcup_{i=1}^mS_i\subsetneqq I^n$, there is $x_0\in I^n-S_i$ for all $i$.
Such $x_0$ satisfies $h_i(x_0)\in S_i^+$ and $h_i^{-1}(x_0)\in S_i^-$, for all $i$.

By Theorem \ref{PP}, $\phi$ is injective and an embedding of $A_\Gamma$ into $nV$.  
\end{proof}

\begin{corollary}
A right-angled Artin group $A_\Gamma$ embeds into $n$-dimensional Thompson groups, where $n$ is the number of complementary edges in $\Gamma$.
\end{corollary}

\begin{proof}
We may assume that every vertex of $\Gamma$ contributes to a complementary edge.
In fact, if we let 
\begin{align*}
V_0(\Gamma)=\{v\in V(\Gamma)\mid \text{ For all $v\not=v'\in V(\Gamma)$, $\{v, v'\}\in E(\Gamma)$.}\},
\end{align*}
then $A_\Gamma=\mathbb{Z}^{|V_0(\Gamma)|}\times A_{\Gamma'}$ for some subgraph $\Gamma'$ satisfying the assumption and $\bar{E}(\Gamma)=\bar{E}(\Gamma')$.
In general, if two groups $G$ and $H$ embed in $nV$, then $G\times H$ again embeds in $nV$.
Therefore, it is enough to consider whether $A_{\Gamma'}$ embeds into $nV$ or not.

Given $A_\Gamma$ satisfying our assumption, we let $V(\Gamma)=\{v_i\}_{1\leq i\leq m}$ be the vertex set and 
$\bar{E}(\Gamma)=\{\bar{e}_k\}_{1\leq k\leq n}$ be the set of complementary edges.
For every $i\in \{1,\ldots, m\}$, we let 
\begin{align*}
D_i=\{k\in\{1,\ldots, n\} \mid \bar{e}_k \text{ contains $v_i$ as an endpoint.}\}.
\end{align*}
We associate $D_i$ with $v_i$.

$\Gamma$ satisfies the condition required in Theorem \ref{Main}, with respect to the subsets $\{D_i\}_{1\leq i\leq m}$ of $\{1,\ldots, n\}$.
\end{proof}


\end{document}